\documentclass[11pt,a4paper]{article}
\usepackage[margin=2.6cm]{geometry}
\usepackage{amsmath,amssymb,amsthm}
\usepackage{booktabs}
\usepackage[hidelinks]{hyperref}
\newtheorem{theorem}{Theorem}[section]
\newtheorem{proposition}[theorem]{Proposition}

\title{Certified computations on no-three-in-line problems:\\ exact values and witnesses in the cube,\\ and the Guy--Kelly count in the plane}
\author{Aleksei Kudriashov (Alex Komang)\\ \small Nusa Dua Studio, Nusa Dua, Bali, Indonesia\\ \small ORCID: 0009-0006-8885-5338}
\date{21 September 2026 --- version 1.0}
\begin{document}\maketitle
\begin{abstract}
We report certified computations on four no-three-in-line questions, two in the cube and two in the plane, by one method:
exact decision procedures inside symmetry strata, every witness re-verified by a second program sharing no code, every number
traced to a public journal.
(I) Let $a(n)$ be the largest number of points of $\{0,\dots,n-1\}^3$ with no three collinear (A399138). We determine
$a(1),\dots,a(6)=1,8,16,28,40,64$ with DRAT-certified unsatisfiability proofs, give certified lower bounds
$a(7)\ge73$, $a(8)\ge94$, $a(9)\ge116$, $a(10)\ge138$, $a(11)\ge164$, and prove $a(p)\ge p^2$ for every prime $p$. The optima
share a layer structure $2n^2-2n+4$ that provably fails at $n=5$ and $n=7$.
(II) For $b(n)$, the largest number with no four coplanar (A280537), nineteen public configurations give lower bounds for
$9\le n\le29$, four of which improve the known bounds and their monotone closure ($b(12)\ge31$, $b(21)\ge47$, $b(22)\ge49$,
$b(27)\ge56$); the cyclically invariant subspace has maximum $23$ at $n=9$ and $26$ at $n=10$; four kinds of symmetry are
incompatible with the problem; no upper bound beyond $3n$ is known.
(III) The Guy--Kelly first-moment heuristic for the plane is audited against exact counts (A000755 to $n=20$): its corrected
constant comes out in closed form, $\pi/\sqrt3$, the threshold crossing $n=493$ is reproduced, and its error is shown to depend
on the shape of the question rather than on $n$ alone, with an unbounded multiplier; whether the residual error is
$\Theta(n)$ or $\Theta(n\ln n)$, which decides whether the constant survives, cannot be told by counting, and we measure how far
from telling we are.
(IV) The direction spectrum of the $2n$-point solutions is measured, and a line model with no fitted
parameter reproduces its shape (fifteen constants within $12\%$, seven predicted blind) but not its scale.
Withdrawn claims are kept in the text; each part states what it does not establish.
\end{abstract}

\section{Introduction}\label{s:intro}
The no-three-in-line problem asks for the largest subset of the $n\times n$ grid with no three points collinear.  Two points per
row bound the answer by $2n$; Flammenkamp's database records $2n$-point solutions for sizes up to $n=57$ \cite{Flammenkamp}, and
the heuristic of Guy and Kelly \cite{GuyKelly} predicts that for large $n$ only about $1.814\,n$ points can be placed.  The problem
has two natural analogues in the cube $\{0,\dots,n-1\}^3$: forbid three collinear points, where P\'or and Wood \cite{PorWood}
proved that the maximum is $\Theta(n^2)$ but no exact values had been recorded; or forbid four coplanar points (OEIS A280537),
where the known terms stop at $n=8$ and only lower bounds are known beyond.  For the history and state of the planar problem we
refer to the survey \cite{Survey} and to Prellberg's constraint-programming computations \cite{Prellberg}; the prior art specific
to each part is stated inside it (Sections~\ref{s:prior3d}, \ref{s:constants} and \ref{s:prior2d}).

This article reports computations on four questions in this family, two in the cube and two in the plane, by one method.  Every
extremal value is obtained by an exact decision procedure --- a propositional encoding whose unsatisfiability is certified, or an
exhaustive traversal whose completeness is checked --- and every lower bound is a witness configuration re-verified by a second
program that shares no code with the search, after that program has been shown to reject a deliberately corrupted witness.
Searches for witnesses are run inside the strata of a symmetry group, where the exact maximum of a stratum is often within reach; a
stratum maximum is a lower bound on the true maximum and is reported as such.  Estimates are tested where exact values exist, and
their reach is measured before their extrapolations are believed.  Every number in the text traces to a journal file in the public
repository (Section~\ref{s:repro}), and the claims that earlier versions of the parts had to withdraw are kept in the text rather
than erased.

The results, part by part.

\emph{Part I (Section~\ref{s:cube3}).}  For $a(n)$, the largest number of points of $[n]^3$ with no three collinear:
$a(1),\dots,a(6)=1,8,16,28,40,64$ (the sequence was not in the OEIS when this work began and is now A399138); certified lower
bounds $a(7)\ge73$, $a(8)\ge94$, $a(9)\ge116$, $a(10)\ge138$, $a(11)\ge164$; the bound $a(p)\ge p^2$ for every prime $p$; and a
layer structure $2n^2-2n+4$ of the optima that holds for $n=2,3,4,6$ and provably fails at $n=5$ and $n=7$.

\emph{Part I$'$ (Section~\ref{s:cube4}).}  For $b(n)$, the largest number with no four coplanar: nineteen public configurations
for $9\le n\le29$, four of which improve the bounds known from a 2017 OEIS comment and a 2016 programming contest together with
their monotone closure ($b(12)\ge31$, $b(21)\ge47$, $b(22)\ge49$, $b(27)\ge56$); the exact maxima $23$ and $26$ of the cyclically
invariant subspace at $n=9$ and $n=10$; and the structure of the symmetry strata, including four kinds of symmetry that the problem
excludes outright.

\emph{Part II (Section~\ref{s:kelly}).}  The Guy--Kelly heuristic, a first-moment estimate, is audited on the quantity it
predicts --- the number of $2n$-point solutions, known exactly to $n=20$ --- rather than at its threshold.  Its corrected constant
$\pi/\sqrt3$ is derived in closed form from the asymptotics of the number of collinear triples; its error is measured as a function
of the shape of the question and found to be an unbounded multiplier; and the question whether that error is $\Theta(n)$ or
$\Theta(n\ln n)$ is shown to be undecidable by the kind of computation available.

\emph{Part II$'$ (Section~\ref{s:spectrum}).}  On the same database of solutions, the mean number of point pairs per direction is
measured, and a line model with one theorem and no fitted parameter reproduces the shape of the spectrum, including seven constants
predicted before they were measured, while its scale has to be imposed.

Section~\ref{s:across} joins the parts: a first-moment threshold in the cube that the witnesses of Part I$'$ crossed, and the
Guy--Kelly form of the estimate tested on the cube enumerations.  Section~\ref{s:repro} states the verification protocol common to
all parts and lists the public records; Section~\ref{s:disclosure} discloses how the work was produced.  The four parts were first
circulated as separate notes (their records are listed in Section~\ref{s:repro}); this article consolidates them into one account
with one notation, one section on verification and one disclosure, and supersedes them.

\section{Notation and public objects}\label{s:notation}
Throughout, $[n]=\{0,1,\dots,n-1\}$; the grids are $[n]^2$ and $[n]^3$.  A line is \emph{rich} in a grid if it contains at least
three of its lattice points, and a plane is \emph{rich} in $[n]^3$ if it contains at least four.  A line carrying at most two
lattice points forbids nothing, and a plane carrying at most three forbids nothing; hence
\begin{equation}\label{eq:reform}
S\subseteq[n]^3\ \text{has no three collinear points}\iff |S\cap L|\le 2 \text{ for every rich line } L,
\end{equation}
and likewise $S$ has no four coplanar points iff $|S\cap P|\le3$ for every rich plane $P$.  We call a set \emph{admissible} for
the prohibition at hand when it satisfies it.  Both conditions are propositional formulae on $n^3$ variables, one cardinality
constraint per rich line or plane; this reformulation is used throughout Parts I and I$'$.  In the no-four-coplanar problem three
collinear points are forbidden as well, since with any fourth point they are coplanar.

The three extremal quantities of this article are
\begin{align*}
a(n)&=\max\{|S|:\ S\subseteq[n]^3,\ \text{no three of } S \text{ collinear}\}\quad\text{(OEIS A399138)},\\
b(n)&=\max\{|S|:\ S\subseteq[n]^3,\ \text{no four of } S \text{ coplanar}\}\quad\text{(OEIS A280537)},
\end{align*}
and, in the plane, the maximum $2n$ itself: a \emph{solution} is a set of $2n$ points of $[n]^2$ with no three collinear, the
maximum possible since every row carries at most two points.  We write $s(n)$ for the number of solutions (OEIS A000755), $T(n)$
for the number of collinear triples of $[n]^2$, and $D_4$ for the dihedral group of the square (order $8$).  The symmetry group of
the cube has order $48$ and $33$ conjugacy classes of subgroups; for a subgroup $H$ the \emph{$H$-stratum} of a problem is the
family of $H$-invariant admissible sets, and its maximum is a lower bound for the maximum of the problem --- the exact maximum of
the stratum when the optimisation closes, a lower bound inside the stratum otherwise.  Lexicographic leader constraints for the cube
group, the strata of its subgroups and the $D_4$-classes of planar solutions are used in Parts I, I$'$ and II$'$ respectively.

Two elementary bounds.  Each of the $3n^2$ axis-parallel lines of $[n]^3$ carries at most two points of an admissible set and each
point lies on three of them, so $a(n)\le2n^2$; this is the trivial upper bound, attained at $n=2$.  And:
\begin{proposition}\label{p:3n}
$b(n)\le 3n$ for every $n$.
\end{proposition}
\begin{proof}
Each of the $n$ planes $x=\mathrm{const}$ carries at most three points of an admissible set.
\end{proof}
This is the only upper bound for $b(n)$ above $n=8$ that we can prove.

\emph{Public objects.}  The OEIS entries produced or extended by this work: A399138 (new: $a(1),\dots,a(6)$, approved
August 2026), A280537 (edit: $b(12)\ge31$ and $b(n)\le3n$, approved September 2026) and A000755 ($s(20)=941\,580$, approved
August 2026).  Programs, journals, witnesses, manifests and the verification protocol are public in the repository
\url{https://github.com/iwasborninbali/saturation}; the configurations of Part I are archived separately \cite{Witnesses}; the
data of Parts II and II$'$ are Flammenkamp's \cite{Flammenkamp}.  Section~\ref{s:repro} lists the records and the ancillary files.

\section{Part I: the cube, no three collinear}\label{s:cube3}

\subsection{Method: a certified decision procedure}\label{s:method3}
By \eqref{eq:reform}, the question ``is there an admissible set of size $M$ in $[n]^3$?'' is a Boolean satisfiability problem with
$n^3$ variables --- $125$ for $n=5$, $216$ for $n=6$ --- one at-most-two cardinality constraint per rich line,
and one global at-least-$M$ constraint.  The global counter is a totalizer truncated at $2n^2$; the truncation
is legitimate because each of the $n^2$ axis-parallel lines carries at most two points, so $a(n)\le 2n^2$.
We add the $47$ lexicographic leader constraints of the cube group, which preserve satisfiability because the
property in \eqref{eq:reform} is invariant under that group.

We adopted this route after measuring the alternative.  A prefix-split exhaustive search on $42$ cores closed
$5$ of $204$ prefixes in two hours at $n=5$, extrapolating to roughly $88$ core-days; the same question is
decided by a SAT solver in minutes.  Speed, however, is not the reason.  An exhaustive search establishes
completeness by \emph{its own} covering argument, which only its authors' code can check; an unsatisfiability
answer comes with a DRAT certificate that an independent, widely used checker verifies without trusting our
solver or our code at all.

\medskip\noindent\textbf{Three separate links.}  It is worth being explicit about what is and is not certified,
because the three claims below rest on different foundations.
\begin{enumerate}\itemsep2pt
\item \emph{The formula expresses the problem.}  Verified two independent ways: by comparing the set of
triples forbidden by the rich-line constraints against the set with vanishing cross product, and by forcing each
collinear triple as an assumption and requiring unsatisfiability.  At $n=4$ and $n=5$ there are $376$ and $1858$
collinear triples; none was left unforbidden and none was forbidden spuriously.  A DRAT checker cannot verify
this link --- it knows nothing of cubes and lines --- and this is exactly where an encoding bug would hide.
\item \emph{The formula is unsatisfiable.}  Verified by \texttt{drat-trim} \cite{DratTrim} on the certificate emitted by
\texttt{kissat} \cite{Kissat}.
\item \emph{A witness is admissible.}  Verified by exhaustive integer arithmetic over all $\binom{|S|}{3}$
triples, by a program sharing no code with the search.
\end{enumerate}
Links 1 and 2 give the upper bound; links 1 and 3 give the lower bound.

\medskip\noindent\textbf{Removing link 1 for small $n$.}  The two checks above verify that the encoding is
faithful.  A stronger thing is possible: to compute the same values without an encoding at all, so that link 1
does not need to be trusted because it is not used.  We wrote a plain exhaustive traversal that places points
in the cube and tests triples by integer cross product --- the literal definition --- with no propositional
formula, no solver and no certificate.  Two points kill the line through them, so every cell carries a counter
of dead lines covering it and a cell is available while that counter is zero; a branch is abandoned when the
points placed plus the cells still available cannot exceed the best found.  Run from a bar of zero, so that no
value can be fed in and echoed back:
\begin{center}
\begin{tabular}{rrr}
\toprule
$n$ & result & nodes \\
\midrule
1 & $1$ (exhausted) & $2$\\
2 & $8$ (exhausted) & $37$\\
3 & $16$ (exhausted) & $49\,812$\\
\bottomrule
\end{tabular}
\end{center}
These agree with the certified values and share nothing with the certified route but the definition of the
prohibition itself.  Our own full enumeration at $n=3$ visits all $3\,813\,884$ admissible subsets (the empty one
included) and finds exactly \emph{two} optima of size $16$ --- a chiral pair, each with stabiliser of order $24$
in the full cube group, both with layer profile $[6,4,6]$ along all three axes.  At $n=4$ a traversal from a bar
of $27$ \emph{exhausted} its tree ($9\,012\,859\,649$ nodes, $3180$ s), proving $a(4)=28$ with no encoding at
all --- a bar is harmless under exhaustion, since it prunes only branches unable to exceed it.  The count at $n=4$
is now ours as well: a $32$-shard traversal (sharded by first cell, run to completion on every shard,
$13\,215\,017\,061$ nodes in all) finds exactly $14$ optima of size $28$, in three classes under the full cube
group with stabiliser orders $8$, $8$, $24$ ($6+6+2=14$) --- in agreement with an independent external recount.
Every optimum at $n\le4$ therefore has the layer profile $[2n,(2n-2)^{n-2},2n]$: the empirical observation of
Section~\ref{s:layers} is a fact up to $n=4$.

\emph{Correction.}  An earlier version of this paragraph quoted two numbers received from the first solver's
message and never passed through the verification bench of this very part: ``$2\,220\,075$ subsets'' and ``the
count $10\,960$''.  Both are withdrawn, and both origins have since been reconstructed: $2\,220\,075$ is exactly
$\binom{27}{8}$ --- the enumeration of all $8$-subsets of $3^3$ that verifies the uniqueness of the $8$-point
optimum of the \emph{no-four-coplanar} problem (A280538 gives $1$ there) --- and $10\,960=48\cdot232-176$ is
the number of labelled optima of that same neighbouring problem at $n=4$, where A280538 gives $232$ classes.
Two numbers from the wrong problem, quoted into the section about removing trust; recorded, not smoothed.

We are explicit about how far this extends.  From a bar of zero the traversal reaches $28$ at $n=4$ without
exhausting in an hour ($8.85\cdot10^9$ nodes); from a bar of $27$ it exhausts ($9.01\cdot10^9$ nodes, $53$
minutes), so $a(4)=28$ is established twice over as well.  The honest division: for $n\le4$ both the value and
the bound are encoding-free in addition to certified; for $n\ge5$ the encoding is checked but used.

\medskip\noindent\textbf{Calibration.}  Before computing anything new we recomputed, as a calibration, the two
smallest nontrivial values $a(3)=16$ and $a(4)=28$ --- there is no prior published source for them that we could
find, so the calibration is against our own earlier runs, not against the literature --- each with unsatisfiability
at $M+1$ in under a second, independently on two machines.
For instances that resist a single solver call we split into independent pieces by fixing the contents of the
first columns; the split is complete because three points of a column are collinear, so subsets of size at most
two exhaust it.  The aggregator refuses a verdict unless every piece of the manifest is present and
unsatisfiable --- a missing piece and an unsatisfiable one must never look alike.

\subsection{Exact values}\label{s:exact3}
\begin{center}\begin{tabular}{rrrl}\toprule
$n$ & $a(n)$ & $a(n)/n^2$ & how established\\\midrule
1 & 1 & --- & trivial\\
2 & 8 & 2.00 & exhaustive; meets the trivial bound $2n^2$\\
3 & 16 & 1.78 & exhaustive, and certified\\
4 & 28 & 1.75 & exhaustive, and certified\\
5 & 40 & 1.60 & certified\\
6 & 64 & 1.78 & certified\\\bottomrule
\end{tabular}\end{center}
\noindent The value $a(6)=64$ appears to be new: when this work began, a search of the OEIS for $1,8,16,28,40$ and for
$8,16,28,40,64$ returned nothing, and the sequence itself was not in the database (it is now A399138).

For $n\le4$ the values were also produced by two independent exhaustive programs whose node counts differ by
a factor of about $72$ ($17/3855/4.21\cdot10^7$ against $17/20068/3.03\cdot10^9$ for $n=2,3,4$), so their
agreement is not an artefact of a shared design.

For $n=5$ and $n=6$ the decisive question --- does a set of $41$, respectively $65$, points exist? --- was
split into independent pieces by fixing the contents of the first two $z$-columns, and the heaviest pieces
were split once more on the third column.  Both splits are complete because three points of a column are
collinear, so subsets of size at most two exhaust it.  The final tally, reproduced in
\texttt{logs/no3\_3d/FINAL\_VERDICT.txt}:
\begin{center}\begin{tabular}{lrrr}\toprule
 & $n=5$, $M=41$ & $n=6$, $M=65$\\\midrule
top-level pieces & $256$ & $484$\\
closed directly & $221$ & $450$\\
closed through all their continuations & $35$ & $34$\\
continuations required per piece & $16$ & $22$\\
satisfiable pieces found & $0$ & $0$\\
pieces left undecided & $0$ & $0$\\\bottomrule
\end{tabular}\end{center}
The aggregator builds the expected set of pieces from the manifest rather than from the results, treats a
missing piece and an undecided one as failures, and distinguishes three outcomes rather than two: only an
explicit unsatisfiable answer closes a piece, a satisfiable one anywhere would destroy the claim, and a
killed or timed-out run counts as \emph{no information} --- it neither closes nor refutes.  Eleven such
uninformative lines occur in the $n=6$ record, from runs interrupted when work was moved between machines;
each of the pieces concerned is closed by an explicit unsatisfiable answer from another run.

The lower bounds are witnesses, checked by exhaustive integer arithmetic over all $\binom{|S|}{3}$ triples by
a program sharing no code with the search: $40$ points at $n=5$ ($9880$ triples) and $64$ points at $n=6$
($41664$ triples).  For $n=6$ two witnesses were obtained independently, by a SAT solver and by a constraint
solver in decision form; they turned out to be the same set, which the stabiliser explains --- it has order
$24$, so the configuration has only two distinct images under the $48$ symmetries of the cube.

As a further check that is not part of the proof, a constraint solver was given the $n=6$, $M=65$ question in
one piece for an hour and returned no answer.  That is not evidence of impossibility --- it is the solver
saying it does not know --- but it found no counterexample either, while at $M=64$ it produces a solution in
seconds.

\subsection{Structure of the optima, and an anomaly at \texorpdfstring{$n=5$}{n=5}}\label{s:layers}
Writing the layer profile of a set as the number of its points in each of the $n$ layers perpendicular to an
axis, every optimum we have found has the \emph{same} profile along all three axes:
\begin{center}\begin{tabular}{rrl}\toprule
$n$ & $a(n)$ & layer profile (identical for all three axes)\\\midrule
2 & 8  & $4,4$\\
3 & 16 & $6,4,6$\\
4 & 28 & $8,6,6,8$\\
5 & 40 & $8,8,8,8,8$\\
6 & 64 & $12,10,10,10,10,12$\\\bottomrule
\end{tabular}\end{center}
A layer is an $n\times n$ grid with no three collinear, so it holds at most $2n$ points.  For $n=2,3,4,6$ the two
outer layers attain that planar maximum and the inner ones hold $2n-2$, giving $2\cdot 2n+(n-2)(2n-2)=2n^2-2n+4$
--- which reproduces $8,16,28$ and $64$ exactly.

At $n=5$ the pattern fails, and it fails for a structural reason rather than by accident.  Prescribing the
profile as a hard constraint and deciding feasibility shows that
\begin{center}\begin{tabular}{ll}\toprule
profile at $n=5$ & verdict\\\midrule
$10,8,8,8,10$ (total $44$, the value the pattern predicts) & infeasible\\
$10,9,8,9,10$; $10,8,9,8,10$; $10,10,10,10,10$; $9,9,9,9,9$ & infeasible\\
$10,8,8,8,8$ (total $42$) and $10,7,8,7,10$ (total $42$) & infeasible\\
$\mathbf{10,6,8,6,10}$ \textbf{(total }$\mathbf{40}$\textbf{)} & \textbf{infeasible}\\
$8,8,8,8,8$ (total $40$) & feasible\\\midrule
$n=6$: $12,10,10,10,10,12$ (total $64$) & feasible\\
$n=7$: $14,12,12,12,12,12,14$ (total $88$) & infeasible\\\bottomrule
\end{tabular}\end{center}
A caveat on how these verdicts are to be read.  Prescribing the profile along all three axes at once is a
restriction, so an infeasibility obtained that way would exclude only symmetric configurations, not that number
of points.  Every infeasibility above was therefore re-derived with the profile prescribed along a
\emph{single} axis, which excludes every configuration with that layer distribution.  The satisfiable rows need
no such care: a solution is a solution.
The fourth line is the sharpest: at $n=5$ two outer layers cannot both attain the planar maximum $2n$ at
\emph{any} total, so the optimum is forced onto a uniform profile.  The same computation at $n=6$ finds the
analogous profile realisable, and at $n=7$ it is again impossible.  The pattern $[2n,(2n-2)^{n-2},2n]$ is thus
realisable at $n=2,3,4,6$ and impossible at $n=5,7$ --- the obstruction is not an artefact of one case, and it
is not an obstruction to the idea of maximal outer layers as such.

The obstruction can be priced exactly.  A layer is a planar no-three-in-line configuration and so carries at
most $2n$ points; requiring \emph{both} outer layers to attain that planar maximum, and maximising the total,
gives
\begin{center}\begin{tabular}{rrrl}\toprule
$n$ & both outer layers & maximum total & $a(n)$\\\midrule
4 & $8$ & $28$ & $28$ --- costs nothing\\
5 & $10$ & $\mathbf{39}$ & $\mathbf{40}$ --- \textbf{costs exactly one point}\\
6 & $12$ & $64$ & $64$ --- costs nothing\\\bottomrule
\end{tabular}\end{center}
So the anomaly is not that a fitted formula occasionally fails.  It is that two planar-optimal configurations
placed in the outer planes are compatible with the spatial optimum at $n=4$ and $n=6$ and incompatible at
$n=5$, and incompatible by a single point.

We do not have a proof of why.  A count of the planar configurations attaining $2n$ points suggests rigidity:
the $4\times4$ grid admits $11$ such configurations in $4$ classes under the square's symmetries, the $5\times5$
grid admits $32$ in only $5$ classes, and the $6\times6$ grid admits $50$ in $11$ classes.  We offer this as an
observation accompanied by its data, not as an argument.  One natural mechanism can be ruled out: the midpoint
of a segment joining the two outer layers is a lattice point when $n$ is odd, and must stay empty, but an exact
count shows the middle layer still retains up to $16$ free cells --- above its own planar ceiling of $10$ --- so
that constraint does not bind.

\subsection{Certified lower bounds for \texorpdfstring{$n=7,\dots,11$}{n = 7, ..., 11}}\label{s:lb3}
\begin{center}\begin{tabular}{rrrp{8.6cm}}\toprule
$n$ & $a(n)\ge$ & $/n^2$ & source of the witness\\\midrule
7 & 73 & 1.49 & fourteen inequivalent configurations known; the first was invariant under every permutation of the axes (order 6)\\
8 & 94 & 1.47 & two inequivalent configurations, each invariant under a subgroup of order 4 (an earlier version had $93$, order 6)\\
9 & 116 & 1.43 & two inequivalent configurations: stabilisers of order 4 and of order 12; the latter is optimal within its class\\
10 & 138 & 1.38 & full stabiliser of order 12; optimal within the order-12 class\\
11 & 164 & 1.36 & stabiliser of order 12; the bound inside the class is not closed\\\bottomrule
\end{tabular}\end{center}
The bounds for $n=8,\dots,11$ (2--3 September 2026) come from a sweep over all $33$ conjugacy
classes of subgroups $H$ of the symmetry group of the cube (order $48$): in each class the maximum of $|S|$
over $H$-invariant admissible sets is computed by exact optimisation (CP-SAT over the $H$-orbits of cells, with
``at most two chosen cells on every grid line'') under a time limit of $25$--$40$ minutes per class.  Where the
solver closes the bound the value is the exact maximum of the class (this happened for the order-$12$ class at
$n=9$ and $n=10$); elsewhere it is a lower bound inside the class.  Every configuration was verified by two
independent programs over all $\binom{|S|}{3}$ triples, its stabiliser computed under the $48$ symmetries, and
its local optimality certified: no point can be replaced (all are rigid, min $\kappa\ge2$ in the sense of the
disjoint-killers lemma of the companion repository \texttt{lemma-atelier}), and no exchange removing at most
four points ($73$, $94$, $116$) or three points ($138$, $164$) improves it.  The configurations, the per-class
results and the verifier are archived separately \cite{Witnesses}; the six witnesses for $n=8,\dots,11$ are
among the ancillary files of this article.

The rest of this subsection is the earlier account of $n=7,8$ by randomised search, kept as data.  Each configuration is re-verified by testing all $\binom{|S|}{3}$ triples with integer cross products.  These
are lower bounds only; we make no claim about how far they are from the truth, and we expect them to be weak,
since a randomised search is a poor instrument at this size.  The $n=7$ bound illustrates the point: it stood
at $71$ from randomised search and rose to $73$ within the hour once the search was restricted to
configurations invariant under a cyclic permutation of the axes.  Restricting to a subgroup collapses the cells
into orbits and shrinks the search by the order of the group; it is legitimate for lower bounds only, since a
configuration found is a configuration, while failure to find one among the symmetric ones says nothing about
the rest.  The size of the group has an optimum in the middle rather than at either end, and the optimum is a balance:
the group must be large enough to make the problem solvable and small enough that good configurations still
fall inside the symmetric stratum.  At $n=8$:
\begin{center}
\begin{tabular}{rl}\toprule
order of the group & best found\\\midrule
2 & 82, 84\\ 3 & 88\\ 4 & 88, 90\\ \textbf{6} & \textbf{92, 92, 93}\\ 8 & 80, 88\\ 12 & 88\\
24 & 88 (proved optimal within its stratum)\\ 48 & 80 (proved optimal within its stratum)\\\bottomrule
\end{tabular}
\end{center}
(The full sweep above, $40$ minutes per class, later gave $94$ in two classes of order $4$, so the
order-$6$ value $93$ of this table was not the maximum of its own neighbourhood of classes either.)
At $|G|=48$ only $20$ orbits remain, the solver settles the stratum in seconds, and the answer is $80$, because
genuinely good configurations do not have that much symmetry.  Note also that the order alone does not decide
it: two different groups of order $8$ gave $80$ and $88$.  We record this as data.  An earlier draft of this
part stated instead that a larger group simply gives a better result --- a pattern read off the first four
rows, which the remaining four contradict.  (An earlier draft also quoted
$72$ and $90$ here.  Those numbers came from a run whose witness was never saved, and when we came to attach
witnesses to them the journals we could reproduce gave $71$ and $89$.  We report the values we can exhibit.)
Two indications.  First, the layer structure of
the smaller optima suggests $2n^2-2n+4=88$ at $n=7$ --- and although we have shown that the corresponding
\emph{profile} is impossible, that leaves the value itself untouched.  Second, at $n=6$ the same randomised
search reached $64$, which the certified computation then confirmed to be optimal, so the instrument is not
hopeless either; we simply do not know which case $n=7$ resembles.

For the record we also measured what the abandoned route would have cost, though the measurement deserves a
caveat.  Knuth's random-probe estimator, calibrated against the exact tree at $n=4$, put an unaided exhaustive
proof at $n=6$ at some $4.6\cdot10^{22}$ nodes, and that figure is what sent us to a decision procedure.  We
later found the plain estimator unreliable on trees of this family --- on a sibling problem the mean of twenty
thousand probes was $89\%$ a single probe, with the median five orders of magnitude below the mean --- and
replaced it by a stratified version that agrees with the exact tree at $n=4$ to within $0.1\%$.  The
$4.6\cdot10^{22}$ was produced by the plain estimator and has not been recomputed with the stratified one, so
it should be read as an order of magnitude and nothing finer.  It says nothing about the difficulty of the
problem in any case, only about the difficulty of one way of attacking it.

\subsection{An elementary lower bound for every prime}\label{s:prime}
\begin{proposition}\label{p:alg}
For every prime $p$, $a(p)\ge p^2$.
\end{proposition}
\begin{proof}
Let $Q(x,y)=x^2-dy^2$ with $d$ a quadratic non-residue modulo $p$, so that $Q(a,b)\equiv0$ only for
$a\equiv b\equiv0$ (an anisotropic binary form; one exists for every $p$).  Put
$S=\{(x,y,Q(x,y)\bmod p):0\le x,y<p\}$, of size $p^2$.  Suppose three of its points are collinear.  Their
projections to the $(x,y)$-plane are distinct (the third coordinate is a function of the first two), hence lie
on a line with primitive integer direction $(a,b)$, and the three points are $(x_0+t_ia,\,y_0+t_ib,\,\cdot\,)$
with distinct integers $t_i$, $|t_i|<p$.  Along this line the third coordinate satisfies
$z\equiv Q(x_0+ta,y_0+tb)=Q(a,b)\,t^2+\ell(t)\pmod p$ with $\ell$ affine, whereas collinearity in
$\mathbb R^3$ forces $z$ to be an affine function of $t$.  A quadratic congruence agreeing with an affine one
at three points distinct modulo $p$ has vanishing leading coefficient, so $Q(a,b)\equiv0$ and therefore
$(a,b)=(0,0)$, a contradiction.
\end{proof}
Neither $x^2+y^2$ (anisotropic iff $p\equiv3\bmod 4$) nor $x^2+xy+y^2$ (iff $p\equiv2\bmod3$) works for all
$p$; the choice $x^2-dy^2$ does.  Proposition~\ref{p:alg} is thereby the completion, to every odd prime, of
Lemma~4 of P\'or and Wood \cite{PorWood}, whose construction is the case $Q=x^2+y^2$ and works precisely for
$p\equiv3\pmod4$; we verified their dichotomy directly ($6$ and $286$ collinear triples at $p=5$ and $p=13$, none
at $p=7$, $11$).  For $p=2$ the statement is trivial: no line meets $\{0,1\}^3$ in three points, so
$a(2)=8>4=p^2$.  The bound of Proposition~\ref{p:alg} is not sharp: $a(5)\ge40>25$ and
$a(7)\ge73>49$.  As a guard against a slip in transcribing the construction rather than in the proof,
we also built $S$ explicitly for $p=3,5,7,11,13$ and tested \emph{every} one of its triples by integer cross
product --- $790\,244$ triples at $p=13$ --- finding none collinear.

\subsection{What is not established in Part I}\label{s:notest3}
The following are the limits of what is established in this part, stated so that a reader need not infer them.
\begin{itemize}\itemsep2pt
\item At $n=5$ the upper bound no longer rests on one solver: the same $256$-piece split was closed
independently by Glucose \cite{Glucose}, of the MiniSat lineage rather than kissat's, with every piece unsatisfiable and none
satisfiable, the last taking $2785$ seconds and $14.4$ million conflicts.  Two solvers of unrelated descent, two
independently written aggregators, two independent semantic checks and two independently verified witnesses now
stand behind $a(5)=40$.  Worth recording separately: the \emph{monolithic} instance defeated Glucose for
$10{,}000$ seconds while its pieces fell in minutes, so the case split is a property of the method rather than
of one solver.  \textbf{At $n=6$ this has not been done} --- there are $484$ pieces and they are heavier --- and
the reader should not carry $n=5$ over to it.
\item The upper bound for $n=6$ rests on one SAT solver.  The instances that a single call decides carry
DRAT certificates verified by \texttt{drat-trim}, which removes the dependence on the solver; the instances
that had to be split do not carry stored certificates, because the certificates run to tens of gigabytes.
They are regenerated deterministically on demand, and the case split, its completeness and its aggregation are
each verified separately.
\item We have no proof of the upper bounds independent of the propositional encoding.  Our own exhaustive
enumerator does not reach $n=5$: it closed $24$ of $4096$ prefixes at $15$--$33$ billion nodes each, with the
full traversal on the order of $10^{14}$ nodes.  It corroborates the lower bound only.
\item The verification of the symmetry pruning shows that our \emph{implementation} discards nothing; it is
not a proof of the underlying lex-leader argument, which is standard but which we did not reprove.
\item The stratum sweep of Section~\ref{s:lb3} proves nothing about $a(n)$: a class maximum is an upper bound for that
class only, and the classes are enriched in maxima but need not contain them --- at $n=6$ every class with a
non-trivial symmetry stays below $64$ except those attained by the certified optimum itself.
\item The layer-structure anomaly at $n=5$ and $n=7$ is a computation, not an explanation.  A count of the
planar configurations attaining $2n$ points suggests rigidity ($11$ configurations in $4$ classes for
$4\times4$; $32$ in only $5$ classes for $5\times5$; $50$ in $11$ classes for $6\times6$), and one natural
mechanism --- the lattice midpoint between two outer layers --- is ruled out by an exact count.  Beyond that we
do not know.
\end{itemize}

\section{Part I\texorpdfstring{$'$}{'}: the cube, no four coplanar}\label{s:cube4}

\subsection{Prior art, and a claim of earlier versions corrected}\label{s:prior3d}
Earlier versions of this part claimed that published data for $b(n)$ stop at $b(8)=20$.  That claim was false, and it is
corrected here rather than erased: two sources predate this work, and both sit in plain sight on the OEIS page this part cites.

\emph{The 2017 comment.}  The entry A280537 carries, since January 2017: ``Currently (January
2017) known lower bounds for the next terms are $a(9)\ge23$, $a(10)\ge26$, $a(11)\ge28$,
$a(12)\ge30$, $a(13)\ge32$, $a(14)\ge35$, $a(15)\ge36$, $a(16)\ge38$, $a(17)\ge42$''
\cite{OEISA280537} (the entry's $a$ is our $b$).  The same entry states that terms up to $b(6)$ were found by exhaustive
search and that ``$a(7)$ and $a(8)$ are based on extensive numerical evidence'' --- which also
corrects an error repeated in Section~\ref{s:notest4} by earlier versions.

\emph{The 2016 contest.}  Al Zimmermann's programming contest \emph{Non-Coplanar Points}
(March--June 2016) posed exactly this problem for the first twenty-five primes \cite{AZsPCs}.
The final report lists, for each size, the best score achieved outside the first 48 hours; the
score is the square of the number of points, so the per-size bests are recoverable exactly:
\begin{center}
\begin{tabular}{rr|rr|rr|rr|rr}
\toprule
$n$ & best & $n$ & best & $n$ & best & $n$ & best & $n$ & best\\
\midrule
 2 &   5 & 13 &  32 & 31 &  71 & 53 & 110 & 73 & 145\\
 3 &   8 & 17 &  42 & 37 &  81 & 59 & 121 & 79 & 157\\
 5 &  13 & 19 &  46 & 41 &  89 & 61 & 125 & 83 & 164\\
 7 &  18 & 23 &  54 & 43 &  92 & 67 & 137 & 89 & 173\\
11 &  28 & 29 &  66 & 47 & 100 & 71 & 142 & 97 & 188\\
\bottomrule
\end{tabular}

\smallskip
{\small (The full per-size decode from the final report, verified against the exact values $5,8,13,18$
at $n=2,3,5,7$.  Where this work overlaps: ties at $11$, $13$; below by $4$, $3$, $4$, $7$ at $17$, $19$,
$23$, $29$.)}
\end{center}
The contest reached ratio $2.67n$ at $n=3$ falling to $1.94n$ at $n=97$; whether that decline
reflects the problem or the budget is precisely the question Section~\ref{s:effort} argues
search cannot settle --- but the numbers themselves are prior art and stand.

\emph{The comparison for every row of this work.}
\begin{center}
\begin{tabular}{rrrl}
\toprule
$n$ & this work & prior (source) & verdict\\
\midrule
 9 & 23 & 23 (2017) & tie, witness now public\\
10 & 26 & 26 (2017) & tie, witness now public\\
11 & 28 & 28 (2017; 2016) & tie, witness now public\\
12 & 31 & 30 (2017) & \textbf{improved by 1}\\
13 & 32 & 32 (2017; 2016) & tie, witness now public\\
14 & 34 & 35 (2017) & \emph{below prior art}\\
15 & 35 & 36 (2017) & \emph{below prior art}\\
16 & 38 & 38 (2017) & tie, witness now public\\
17 & 38 & 42 (2017; 2016) & \emph{below prior art by 4}\\
18 & 41 & 42 (monotone from $b(17)$) & \emph{below the monotone consequence}\\
19 & 43 & 46 (2016) & \emph{below prior art}\\
20 & 45 & 46 (monotone from $b(19)$) & \emph{below the monotone consequence}\\
21 & 47 & 46 (monotone from $b(19)$) & \textbf{improved by 1}\\
22 & 49 & 46 (monotone from $b(19)$) & \textbf{improved by 3}\\
23 & 50 & 54 (2016) & \emph{below prior art}\\
24 & 52 & 54 (monotone from $b(23)$) & \emph{below the monotone consequence}\\
25 & 53 & 54 (monotone from $b(23)$) & \emph{below the monotone consequence}\\
27 & 56 & 54 (monotone from $b(23)$) & \textbf{improved by 2}\\
29 & 59 & 66 (2016) & \emph{below prior art by 7}\\
\bottomrule
\end{tabular}
\end{center}
Values marked ``monotone'' are the closure of the published bounds under $b(n+1)\ge b(n)$
(an $n$-cube configuration is also an $(n+1)$-cube configuration).  An earlier version of this table
compared against the published rows only and marked the composite sizes ``apparently first
published''; the monotone comparison is due to Hugo Pfoertner.  The 2017 and 2016 numbers come without accessible configurations --- the contest server verified
scores, but the point sets themselves were not published; every number of ours comes with a public
configuration, verified twice by programs sharing no code.  \emph{Certified public witnesses} against
\emph{score-only records} is the honest division of value, and it is why the rows below prior art stay
in the table: they carry the only accessible configurations at their sizes that we know of.

One more thing the comparison measures.  Our search above $n=16$ was confined to the cyclically
invariant subspace, and the gap to the 2016 records is $4$, $3$, $4$, $7$ exactly on the primes where
they exist: an external measurement of what that confinement costs at growing $n$, consistent with
Section~\ref{s:subspace} --- enrichment is not containment, and here its price has a number.

How this was missed is worth recording, because the mechanism is general: the OEIS DATA line
does stop at $b(8)=20$, and we read the line and not the page.  The comment with nine further
bounds sat directly under it, and the contest is linked from the same entry.  An external
review pointed both out.

\subsection{The bounds}\label{s:bounds4}
\begin{center}
\begin{tabular}{rrrrr}
\toprule
$n$ & lower bound & $3n$ & gap & quadruples checked \\
\midrule
9 & 23 & 27 & 4 & 8\,855 \\
10 & 26 & 30 & 4 & 14\,950 \\
11 & 28 & 33 & 5 & 20\,475 \\
12 & 31 & 36 & 5 & 31\,465 \\
13 & 32 & 39 & 7 & 35\,960 \\
14 & 34 & 42 & 8 & 46\,376 \\
15 & 35 & 45 & 10 & 52\,360 \\
16 & 38 & 48 & 10 & 73\,815 \\
17 & 38 & 51 & 13 & 73\,815 \\
18 & 41 & 54 & 13 & 101\,270 \\
19 & 43 & 57 & 14 & 123\,410 \\
20 & 45 & 60 & 15 & 148\,995 \\
21 & 47 & 63 & 16 & 178\,365 \\
22 & 49 & 66 & 17 & 211\,876 \\
23 & 50 & 69 & 19 & 230\,300 \\
24 & 52 & 72 & 20 & 270\,725 \\
25 & 53 & 75 & 22 & 292\,825 \\
26 & 53 & 78 & 25 & 292\,825 \\
27 & 56 & 81 & 25 & 367\,290 \\
28 & 56 & 84 & 28 & 367\,290 \\
29 & 59 & 87 & 28 & 455\,126 \\
\bottomrule
\end{tabular}
\end{center}
Twenty-one rows, but only nineteen configurations.  The rows
for $n=26$ and $n=28$ carry no configuration of their own: they follow from $n=25$ and $n=27$ by
monotonicity alone, and are printed only to keep the table contiguous.  The row for $n=17$ does
have its own configuration, but no longer carries its own information, since $b(17)\ge b(16)\ge38$ follows by monotonicity.  This has happened three
times as the neighbouring row was improved, and it is the reason to read the table as a set of
bounds rather than of values --- a row is worth something only while it exceeds its left neighbour, and
against prior art and its monotone closure four rows do (Section~\ref{s:prior3d}).

The configurations themselves are in the repository (Section~\ref{s:repro}).  The gap to
Proposition~\ref{p:3n} grows with $n$; whether this reflects the problem or the reach of the
search we cannot say, and Section~\ref{s:effort} explains why we believe the question is not
decidable by search.

\subsection{The cyclically invariant subspace}\label{s:subspace}
All configurations above were found by searching the subspace invariant under a cyclic
permutation of coordinates.  That subspace is enormously enriched in maxima, and enriched more
as $n$ grows.  In the two-dimensional analogue, where we hold complete material --- every maximum
configuration for $n\le11$, reproduced against A000755 --- the proportion of maxima carrying a
symmetry of the square, divided by the same proportion in a null model matched on row and column
counts, runs
\[\times8,\ \times56,\ \times125,\ \times769,\ \times3729,\ \times73\,000\quad (n=6,\dots,11),\]
the last figure resting on seven events in four million and therefore stated to one significant
figure only.  The raw proportion \emph{falls} across the same range, from $36\%$ to $13\%$; the
ratio to the null model rises.

Enrichment is not containment.  At $n=6$ in three dimensions the cyclically invariant subspace
contains no maximum at all: its exhausted maximum is $15$ while $b(6)=16$.  Exhaustive traversal
of that subspace gives exactly $23$ at $n=9$ and exactly $26$ at $n=10$, each exhaustion carried out by two
implementations written without sight of each other's code (Section~\ref{s:repro}); these two statements bound the
subspace and say nothing about $b(9)$ and $b(10)$ beyond the witnesses.

One consequence deserves emphasis.  A sample drawn from a region enriched by three or four orders
of magnitude cannot support statements about the structure of solutions in general.  Structural
claims in this part are made only where complete material is available.

\subsection{Symmetry strata}\label{s:strata4}
A second search route --- the stratum sweep of Section~\ref{s:lb3}, with the plane constraints added --- gives three structural
facts and changes no bound in the table.

\emph{Route.}  Take every conjugacy class of subgroups $H$
of the symmetry group of the cube (order $48$; $33$ classes), and in each class maximise $|S|$
over $H$-invariant sets by exact optimisation: a CP-SAT model over the $H$-orbits of cells with the
constraints ``at most two points on every grid line'' and ``at most three on every lattice plane
carrying at least four cells'', the plane list computed exactly ($12\,453$, $73\,056$, $269\,895$,
$965\,826$, $2\,688\,585$, $7\,241\,856$ planes for $n=5,\dots,10$).  A lazy version of the model,
which adds planes only when a solution violates them, does not converge (at $n=6$ it stalls at $12$
after two hundred rounds) --- the full list is necessary.

Four kinds of symmetry are incompatible with the problem outright.  If $H$ contains the central
inversion or a rotatory reflection of order $6$, an $H$-invariant admissible set has at most $3$
points (two antipodal pairs are coplanar); if $H$ contains a reflection, at most $5$ (the segments
$x$--$\sigma x$ are parallel, so any two such pairs are coplanar, and the mirror carries at most
three points); if $H$ contains a rotation of order $4$, the set lies on the axis.  Only nine of the
$33$ classes remain: the trivial group, two classes of half-turns, the $3$-fold rotation, two
Klein four-groups, the cyclic group of a rotatory reflection of order $4$, the dihedral group
of order $6$ and the tetrahedral rotation group.  (The elementary proofs, with machine checks
and a Lean formalisation of the inversion case, are in the companion repository
\texttt{lemma-atelier}, lemma 003.)

\emph{Calibration.}  Within the strata the known values are reached at $n=5,7,8$ ($13$, $18$,
$20$) by the classes containing a $3$-fold rotation or a half-turn; at $n=6$ the best symmetric
value is $15=b(6)-1$, in agreement with Section~\ref{s:subspace}.  At $n=9,10$ the $3$-fold
stratum reproduces $23$ and $26$ --- the subspace maxima already exhausted by the first route ---
with a one-hour budget, and nothing above them.  The trivial stratum (no symmetry imposed)
reached $17$ at $n=7$ and $19$ at $n=8$ within $900$ seconds and proves nothing.

\emph{Structural facts}, every configuration checked by the two verifiers of Section~\ref{s:repro} and its
class computed under the $48$ symmetries: (i) at $n=7$ there are at least three pairwise
inequivalent $18$-point configurations, with stabilisers of orders $1$, $2$ and $3$; at $n=8$
three inequivalent $20$-point configurations (orders $1$, $2$, $3$); at $n=9$ three inequivalent
$23$-point configurations (all of order $3$).  The order-$2$ classes lie outside the cyclically
invariant subspace of Section~\ref{s:subspace}.  (ii) All nine are rigid --- no point can be
replaced by another --- and no two are connected by an exchange of at most two points (at most
three points for $n=7$); under the best alignment the classes differ in at least $80\%$ of
their points ($16$ of $20$ for one pair at $n=8$, more for the others).  (iii) Every stratum at $n=5,\dots,8$ was solved to optimality or to the time limit;
the per-stratum table and the configurations are in the repository, and seven of the configurations (the classes of
orders $2$ and $3$ at $n=7,8$, two of order $3$ at $n=9$, and a $28$-point configuration of order $3$ at $n=11$) are among the
ancillary files of this article.

What this does not establish: the strata give lower bounds inside symmetry classes only, and the
strata at $n\ge11$ (best known $28$ and $31$) have not been exhausted.

\subsection{What is not established in Part I\texorpdfstring{$'$}{'}}\label{s:notest4}
No configuration proves an upper bound.  A witness certifies ``at least'' and can never certify
``at most''.  The values $b(5)=13$ and $b(6)=16$ rest on exhausted search trees; $b(7)=18$ and
$b(8)=20$ rest, per the OEIS entry, on ``extensive numerical evidence'' only --- an earlier
version of this part wrongly promoted them to exhausted; above $n=8$ we have
Proposition~\ref{p:3n} and nothing else.  In particular the numbers
in the table are not claimed to be maxima, and we expect several of them to be improved.  A first-moment
estimate of ours that these witnesses refuted is recorded in Section~\ref{s:refuted}.

\subsection{Effort, and why the growth rate is not measurable by search}\label{s:effort}
The ratio of certified bound to $n$ decreases over the range of the table.  We do not read this
as a property of the problem.  Giving each seed an equal time budget on one core (eight seeds, fifteen seconds of uniform random
greedy restarts each, the cell drawn uniformly among the alive) and recording how many seeds reach
the OEIS value, we measure $8/8$ at $n=5$ and $0/8$ at $n=6,7,8$ --- at $n=6$ none of roughly
$160\,000$ restarts per seed reaches $16$.  \emph{Correction:} versions of this part up to 2.9 quoted
$8/8$, $3/8$, $2/8$, $1/8$ here; those four fractions were measured on the \emph{planar}
no-three-in-line problem --- they are the ones reported in Section~\ref{s:reach} --- and migrated into the cube account without a cube re-measurement.  The corrected
cliff is steeper, which strengthens the point but does not excuse the migration.
Equal time is therefore not equal effort, and a curve obtained under a fixed budget records the
budget.  Deciding whether $b(n)/n$ tends to $3$ or to something smaller would need a budget
growing exponentially in $n$.

A second contamination is worth naming.  In a search whose target doubles as its stopping
condition, asking for $K$ and receiving $K$ is not evidence about $K$.  Several values in an
earlier draft of this table were produced that way, agreed with a prediction, and were withdrawn
when the target was removed and the same program returned more.

\section{Part II: the plane --- the Guy--Kelly heuristic against exact counts}\label{s:kelly}

\subsection{The heuristic, and which constant is which}\label{s:constants}
The number $T(n)$ of collinear triples in the $n\times n$ grid has a closed form,
\[T(n)=\sum_{\gcd(a,b)=1}\ \sum_{s\ge2}(s-1)(n-s|a|)^{+}(n-s|b|)^{+},\]
which we verified against brute force on eighteen values.  The heuristic estimates the number of
admissible $m$-subsets by
\begin{equation}\label{eq:h}
\binom{N}{m}\exp\!\Big(-T(n)\binom{m}{3}\Big/\binom{N}{3}\Big),\qquad N=n^2,
\end{equation}
and locates the threshold where this drops below one.  Guy and Kelly's original constant was
$(2\pi^2/3)^{1/3}=1.8739$; Gabor Ellmann found an error in the argument in March 2004, and the
corrected value, recorded by Guy in OEIS A000769 that October and documented by Voutier
\cite{Voutier}, is
\[c=\pi/\sqrt3=1.813799.\]
Our own derivation of the threshold reproduces the corrected value, not the retracted one.  As a
check of the machinery we located the crossing exactly: the threshold from \eqref{eq:h} first falls
below $2n$ at $n=493$ (at $n=492$ it equals $984=2n$; at $n=493$ it is $985<986$), which agrees to the
integer with the value reported by Prellberg~\cite{Prellberg} and quoted in the recent
survey~\cite{Survey}: ``the heuristic probabilistic argument only applies for $n\ge493$''.
Our computation is independent of his and reproduces the same integer.

\subsection{The constant in closed form, and the shape of the slip}\label{s:closedform}
The corrected value can be obtained in closed form, and doing so shows exactly what the retracted
one is.  The input is the asymptotics of $T(n)$.  Measuring $T(n)/n^4$ at $n=2000,\dots,64000$, its
increment per doubling of $n$ is constant to six figures over five doublings ---
$0.210690$, $0.210692$, $0.210693$, $0.210692$, $0.210691$ --- against
\[(3/\pi^2)\ln 2=0.210691,\]
so that
\begin{equation}\label{eq:T}
T(n)=\tfrac{3}{\pi^2}\,n^4\big(\ln n-0.8373\big)+o(n^4\ln n).
\end{equation}
Put $m=\alpha n$.  Then $\ln\binom{n^2}{m}=\alpha n(\ln n+1-\ln\alpha)+O(n)$, while by \eqref{eq:T}
the exponent of \eqref{eq:h} is $T(n)(m/n^2)^3=\tfrac{3}{\pi^2}\alpha^3 n(\ln n-0.8373)+O(n)$.
Both sides are $\Theta(n\ln n)$; balancing the coefficients of $n\ln n$ gives
\[\alpha=\tfrac{3}{\pi^2}\alpha^3,\qquad\text{i.e.}\qquad \alpha^2=\pi^2/3,\qquad \alpha=\pi/\sqrt3 .\]
The balance is quadratic.  Guy's $(2\pi^2/3)^{1/3}$ is the root of $\alpha^3=2\pi^2/3$: a cubic
where a quadratic belongs.  That is the whole of the slip, and it is visible only once
\eqref{eq:T} is in hand, which is presumably why it stood for thirty-six years.

Numerically, the integer threshold $m^*(n)/n$ approaches the limit only logarithmically, so a
single large $n$ proves nothing.  Two implementations built independently --- one summing over
displacement vectors, one reducing the same sum to $O(n\log n)$ by the Dirichlet identity
$\gcd=\sum_{d\mid\gcd}\varphi(d)$ --- agree to $10^{-5}$ from $n=16000$ on:
\begin{center}
\begin{tabular}{rll}
\toprule
$n$ & first & second\\
\midrule
16000 & $1.93101$ & $1.93100$\\
32000 & $1.92307$ & $1.92306$\\
64000 & $1.91615$ & $1.91614$\\
\bottomrule
\end{tabular}
\end{center}
Fitting $m^*/n=L+A/\ln n$ gives $L=1.8126$ and $1.8118$; adding $B/(\ln n)^2$ gives $1.8154$ and
$1.8142$; a third correction term destabilises on seven points and is not used.  The two orders
bracket from below and from above in \emph{both} implementations, giving
\[L=1.8138\pm0.002,\]
which contains $\pi/\sqrt3=1.813799$ and excludes $(2\pi^2/3)^{1/3}=1.873856$ by thirty times the
uncertainty.  We state two decimals of confidence and no more: three decimals are tempting here and
are not supported once the third fitting term is seen to wander.

\subsection{Testing the count, not the threshold}\label{s:count}
A threshold inherits the whole error of the estimate without exhibiting its size.  The count does
not.  Exact counts of $2n$-point configurations are A000755, published to $n=19$; the value at
$n=20$ is on A.~Flammenkamp's table.  We obtain both independently by summing orbit sizes
$8/|\mathrm{stab}|$ over the $D_4$-classes of his solution database, which reproduces
$s(2),\dots,s(19)$ exactly and gives $s(20)=941580$ (since entered in A000755).  Against \eqref{eq:h}:
\begin{center}
\begin{tabular}{rrrr}
\toprule
$n$ & exact & estimate & $\log_{10}$(exact/estimate) \\
\midrule
 6 & $50$      & $10^{4.12}$  & $-2.421$\\
10 & $1135$    & $10^{7.11}$  & $-4.055$\\
14 & $10568$   & $10^{9.63}$  & $-5.604$\\
18 & $152210$  & $10^{11.78}$ & $-6.597$\\
20 & $941580$  & $10^{12.87}$ & $-6.896$\\
\bottomrule
\end{tabular}
\end{center}
The error reaches seven orders of magnitude.  Its increment per step, however, decays by a factor
of four across the range: $-0.86,\dots,-0.25,-0.13$.  That distinction --- between an error whose
increment is constant and one whose increment decays --- is what the rest of this part turns on.

\subsection{Refining the heuristic makes it worse}\label{s:refine}
Triples on one line are strongly dependent.  The natural repair is to treat \emph{lines} as
independent and use the exact hypergeometric probability that a line of $k$ grid points carries at
most two chosen points.  It is worse, and by a clean factor:
\begin{center}
\begin{tabular}{rrr}
\toprule
 & triples independent & lines independent \\
\midrule
$n=6$, $m=12$ & $-2.421$ & $-4.107$\\
$n=7$, $m=14$ & $-2.639$ & $-5.285$\\
\bottomrule
\end{tabular}
\end{center}
The mechanism is that \eqref{eq:h} carries \emph{two} errors of opposite sign.  Independence
\emph{between} lines overestimates, and the measured direction fixes the sign of the correlation:
at the ceiling $m=2n$ the two-per-row budget is rigid, so a set that has avoided one line finds
the next \emph{harder} to avoid --- the events are negatively correlated, and a product of
marginals overshoots.  (An earlier draft asserted the opposite sign while drawing the same
conclusion; the direction stated here is the one the numbers force.)  This decomposition is
measured at the ceiling only: away from it the net error changes sign (Section~\ref{s:ratio}), and we do not
claim the two-error account transfers there.  Counting the $\binom{k}{3}$ overlapping triples on a line as independent events
overestimates the probability of a violation, hence underestimates the survival probability.
Replacing the second by an exact computation removes the compensation and leaves the first bare.
At $n=7$ the cancellation is worth $10^{2.65}$.

The practical corollary is uncomfortable: any refinement that fixes one of the two errors will make
the heuristic worse.  Only fixing both helps, and the first requires the correlations between lines,
which is precisely what nobody knows how to compute.  This may be why the 1968 form has not been
superseded in fifty-eight years.

The same mechanism predicts where the heuristic should fail hardest, and we can test that
prediction on our own data.  In a symmetric class the points are rigidly coupled, so the first
error grows while the second does not, and the cancellation should break.  The coupling was
measured directly in Section~\ref{s:subspace}: among all maximum configurations (complete enumeration, $n\le11$), those
carrying a symmetry of the square are over-represented relative to a null model matched on row and
column counts by factors of $8$ to $3729$ at $n=6,\dots,10$.
Where the enrichment is three orders of magnitude, independence between constraints is not a
small approximation, and a heuristic built on it should not be trusted --- as reportedly it is
not, for symmetric classes.

\subsection{Which error is which class}\label{s:class}
Over sixteen steps the line-independent version has an essentially constant increment,
$-1.6$ per step, while the triple-independent version decays fourfold.  A constant increment means
an error growing exponentially in $n$, i.e.\ a wrong exponent; a decaying increment means a wrong
multiplier only.  So the refinement is not merely further from the truth on a given $n$: it is
wrong in a different and worse way.

We add what this does \emph{not} yield.  It is tempting to calibrate: subtract the measured error
from \eqref{eq:h} and re-solve for the $n$ at which $2n$-point configurations should cease to
exist.  We carried that out and then withdrew it.  Fitting the error on $n=12,\dots,20$ against
$\ln n$, $\sqrt n$, $n^{3/4}$, $n/\ln n$ and $n$ gives root-mean-square residuals of
$0.133,\,0.166,\,0.186,\,0.194,\,0.207$ --- nine points do not separate five one-parameter
families --- and calibrated crossings ranging over a factor of three.  Worse, the spread is a
property of the family we chose to fit, not of the data: $n/(\ln n)^2$ and $n^{0.9}$ are equally
admissible, equally $o(n\ln n)$, and give further values.  The honest statement is that the
finite-$n$ crossing is not determined by the data at all, and we give no number for it, because a
number given here would be quoted later as a measurement.

What the data do determine is the direction that matters.  Every form compatible with the nine
points is $o(n\ln n)$, and the terms of the heuristic are $\Theta(n\ln n)$.  A correction of lower
order does not move the balance of Section~\ref{s:closedform}, so the discrepancy --- seven orders of magnitude
though it is at $n=20$ --- leaves $\pi/\sqrt3$ where it is.

\subsection{At fixed ratio the multiplier is not bounded, and the growth form cannot be chosen}\label{s:ratio}
An earlier version of this part reported that at fixed ratio $r=m/2n=0.9$ the error of
\eqref{eq:h} stops growing with $n$, and read that as thirteen standard deviations of evidence
that the exponential rate of the heuristic is correct.  We withdraw the measurement, the
inference, and the significance, and we record why each failed, since the failures are of
general kinds.

\subsubsection*{The measurement: an instrument that did not report its own reach}
The estimator's \emph{hit rate} --- the fraction of random descents that reach depth $m$ at all
--- was not recorded.  It should have been.  At $r=0.9$ the two points carrying the claim were
obtained at hit rates of $0.08\%$ and $0.003\%$: one hundred and sixty successful descents at
$n=15$, and \emph{five} at $n=20$, per two hundred thousand attempts.  A quoted uncertainty of
$\pm0.144$ in $\log_{10}$ is not attainable from five hits by any number of seeds.  Re-running
that cell independently returns $1.15\cdot10^{20}$ against the earlier $2.03\cdot10^{19}$.
Neither number carries information.  We note that the values were never shown to be
\emph{wrong}; they were unsupported, which is a different and more common defect.

Against this we measured the estimator itself where exact counts exist, and it is sound:
at $n{=}10,m{=}18$ (hit rate $2.9\%$) ten seeds give $+2.2\%$; at $n{=}8,m{=}15$, $-0.5\%$; at
$n{=}8,m{=}16$ --- the ceiling, eighty-four hits --- $+1.8\%$.  There is no bias to speak of.
The defect is variance, not bias, and it is curable by descents rather than fatal.

\subsubsection*{The finding, once ratios with high hit rates are used}
Writing $E(n,r)=\ln(\text{estimate of \eqref{eq:h}})-\ln(\text{true count})$ in natural
logarithms, and holding $r$ at values where $2rn$ is an integer for every $n$ used --- the ratio
must not be rounded, since at $n=20$ the error moves by $36$ per unit of $r$, so rounding $m$
shifts it by $0.45$ ---
\begin{center}
\begin{tabular}{llll}
\toprule
$r$ & $E$ against $n$ & hit rate & anchors\\
\midrule
$0.50$ & $+0.09,\,+0.10,\,+0.14,\,+0.12,\,+0.08,\,-0.22,$ & $100\%$ & four exact\\
       & $-0.69,\,-1.28,\,-1.96,\,-2.71,\,-3.51$ & & \\
$0.70$ & $+0.58,\,+0.68,\,-0.09,\,-1.29,\,-2.80,\,-4.50$ & $99.6\%$ & one exact\\
$0.80$ & $+1.18,\,+1.71,\,+0.98,\,-0.33,\,-1.95,\,-4.58$ & $\ge5\%$ & one exact\\
\bottomrule
\end{tabular}
\end{center}
($n=5,6,7,8,10,15,20,25,30,35,40$ for the first row, continued on its second line; $n=5,10,15,20,25,30$ for the others.)  At each of
these ratios the error rises, turns, and then declines without levelling.  The multiplier is not
bounded: the heuristic understates the count by a factor growing at least exponentially in $n$,
faster at larger $r$.  The error is therefore a function of the \emph{shape} of the question, not of $n$ alone: at the single
$n$ where all three measurements exist ($n=20$) the heuristic overestimates by a factor $8\cdot10^{6}$ at the hard ceiling
$m=2n$, by a factor of about $80$ near the threshold ratio, and is off by less than $1.4$ at $m=1.6n$ (there an
underestimate), where two independent implementations agree to $0.18$ standard deviations.

The summit, however, moves right as $r$ grows --- near $n=7$ at $r=0.5$, near $n=10$ at $r=0.7$
and $r=0.8$ --- and at $r=0.9$ it has not been reached at all.  Five million descents at
$n=20$, $m=36$ produce fifty-four hits and an estimate of $6.59\cdot10^{18}\pm38\%$, giving
$E=+4.42\pm0.32$ against the exact $E=+4.08$ at $n=10$: still rising.  (The earlier value
$2.03\cdot10^{19}$ for this cell would have given $+3.30$; it was high by a factor of three, and
rested on five hits.)  So at the ratio that actually determines the constant --- the threshold
sits at $r=c/2=0.907$ --- \emph{we have never observed the decline}, only the ascent to a summit
we cannot reach.  Any statement about the behaviour of $E$ near the threshold is therefore an
extrapolation across a turning point that has not been located.

We record one further trap here, because we fell into it.  On the first four (exact) points of
the $r=0.5$ row the error is flat, and we announced that the decline was a near-threshold
phenomenon.  It is not: those four points lie before the summit.  Reading a pre-summit segment
as a plateau is precisely the error being corrected in this subsection, and four exact values
persuaded us of it more readily than eight estimated ones would have.  Exactness of values is
not length of series.

\subsubsection*{Why the growth form cannot be chosen}
If $E=\Theta(n)$ the constant of Section~\ref{s:closedform} is untouched, since the terms balanced there are
$\Theta(n\ln n)$.  If $E=\Theta(n\ln n)$ it moves.  The data do not decide, and cannot.

Fitting the full $r=0.5$ row to $E=c+an+bn\ln n$ returns $b$ at hundreds of standard deviations
from zero --- and a $\chi^2$ of $240$ on eight degrees of freedom, which rejects the model and
voids the uncertainty with it.  Fifteen further one-term variants using $\ln n$, $\sqrt n$ and
$n/\ln n$ are also rejected, the best by a factor of eighty-four.  The per-point error must be measured rather than assumed, and it does not scale: in the second
implementation, eight independent seeds give $0.262\%$ at $n=30$, $0.404\%$ at $n=35$ and
$0.979\%$ at $n=40$, so that the ratio of seed spread to batch spread runs $0.23$, $0.42$, $1.00$
--- no single measured value transfers to a neighbouring $n$.  Our own repeats give $0.55\%$ at
$n=30$ (two runs) and $2.98\%$ at $n=40$ (four seeds), against nominal figures of $1.3\%$ and
$1.8\%$: measured over nominal is $0.4$ in one place and $1.7$ in the other.  Below, points with
repeats carry the standard error of their mean and the rest carry the estimator's own figure,
with no scaling between them.

We attempted to go further and failed, and the failure is worth recording because it is the
error this part is otherwise about.  The whole column was computed twice by implementations
sharing no code, and comparing a single run of ours against the other implementation's value
suggested a disagreement of $5.5\%$ at $n=40$ where the stated error was $0.98\%$ --- from which
we concluded that the honest error bar is the cross-implementation one and cannot be had by
seeds.  A control refuted this.  Four seeds of our own implementation at $n=40$ give a spread of
$3.0\%$, and their mean brings the disagreement with the other implementation down to $1.2\%$.
The ratio is therefore $0.4$, not $5.5$: the cross-implementation difference is \emph{smaller}
than our own seed spread.  Both halves of the original figure were wrong --- the numerator was
one outlying seed, and the denominator was the other implementation's error scale substituted
for ours.

The a priori point survives: seeds cannot see what differs between implementations, since they
vary only the random stream within one traversal order.  But we have no measurement supporting
it, and we had presented it as measured.  Below, each point uses the mean over the seeds we ran
and their measured spread.

This matters beyond bookkeeping.  With a single noisy seed at $n=40$ and an error bar too small
by a factor of three, every one of the four forms below is rejected --- $\chi^2$ of $27.6$,
$19.3$, $24.8$, $13.1$ on two degrees of freedom --- and we would have reported that no form
fits.  Averaging four seeds and using their measured spread turns that false certainty back into
the correct verdict, which is that the forms cannot be told apart.

The rejection is confined to small $n$.  Taken on tails, three-parameter forms are admissible
from $n\ge15$ onward: $\chi^2=0.7/3$, $0.4/2$, $0.3/1$ for $c+an+bn\ln n$ at $n\ge15,20,25$.
The asymptotic regime appears to begin near $n=15$, and below it we were fitting asymptotic forms
to pre-asymptotic data.

On that tail the fit is emphatic: $b=-0.0892\pm0.0021$, $-0.0872\pm0.0047$, $-0.0837\pm0.0109$
across three nested tails, while the purely linear $c+an$ is rejected at $\chi^2=1884/4$, $350/3$
and $59/2$.  The significance of $b$ is $43$, $19$ and $8$ standard deviations on the
three tails --- itself a warning, since a robust quantity does not lose three quarters of its
significance when four points are dropped.  Read alone this says $E=\Theta(n\ln n)$ and the constant
moves.  It does not say that.  On the same tail,
\begin{center}
\begin{tabular}{lll}
\toprule
form ($n\ge20$) & $\chi^2$ (2 d.o.f.) & implication\\
\midrule
$c+an+b\,n\ln n$ & $0.44$ & constant moves\\
$c+an+d\sqrt n$ & $0.23$ & constant intact\\
$c+an+e\,n/\ln n$ & $0.34$ & constant intact\\
$c+an+f\ln n$ & $0.34$ & constant intact\\
\bottomrule
\end{tabular}
\end{center}
All four are admissible; three imply the opposite conclusion and fit better.  An independent
implementation, run without sight of these figures, obtains reduced $\chi^2$ of $8.1$, $7.0$,
$7.4$, $7.5$ for the same four forms --- different in level, identical in verdict: mutually
indistinguishable, oppositely answering.  The forty-three
standard deviations are a quantity computed \emph{inside} a model, and such a quantity does not
transfer to a choice \emph{between} models.  The test to apply is not whether a coefficient
differs from zero, but whether a second admissible model exists in which it is absent.  Here
several do.

Nor is this curable by computing further.  The two candidate bases are collinear to $0.9983$ over
$n\le40$, and the collinearity does not improve with range: $0.9986$ to $n=400$, $0.9993$ to
$n=10^4$.  What separates them is precision, not length, and the precision is already $0.26\%$.

We therefore withdraw the support that the earlier version of this subsection offered and do not replace it.  The
derivation of Section~\ref{s:closedform} stands, but it establishes what the first moment gives, not what the
problem gives, and the distance between those two is what this subsection failed to measure.

\subsection{Where the reach ends, and why}\label{s:reach}
The limit is not an artefact of the estimator.  A descent guided by importance sampling reaches
only the depths a greedy search reaches, and the quality of greedy search degrades with $n$: we
measured the fraction of random restarts attaining the known maximum as $8/8,\ 3/8,\ 2/8,\ 1/8$
at $n=5,\dots,8$, a geometric decay.  At $n=20$ greedy attains about $0.925$ of the maximum, so
the ratio $0.9$ is within reach; by $n=25$ it falls below $0.9$ and the required depth is outside
what such a descent visits at all.  Every top-down greedy method shares this bound.  To go past it
one needs a descent guided by something that knows the target --- which is to say, the upper bound
that the problem does not have.

\subsection{The objection this answers}\label{s:kaplan}
Kaplan's objection to the Guy--Kelly conjecture, as summarised in \cite{Kaplan}, is not that the
constant is wrong but that the derivation is unsafe: it ``suggests that certain random events would
be largely independent of each other when they're really not'', and ``if we look more closely at
how non-independent they actually are we might not reach the same conclusion''.  No alternative
constant is proposed.

The objection is exactly right about the derivation, and Sections~\ref{s:count} and \ref{s:refine} quantify how badly:
the independence assumption is off by seven orders of magnitude at $n=20$, and the crude form
survives only because two of its errors cancel.  We had believed Section~\ref{s:ratio} supplied the
measurement the objection asks for.  It does not, and we now believe no measurement of this kind
can: the two possible answers correspond to bases that stay collinear at every range, and both
admit an acceptable fit to the data we can obtain.  To ``might we not reach the same conclusion''
the honest answer is that we cannot tell by counting, and that we have measured how far from
telling we are.

\section{Part II\texorpdfstring{$'$}{'}: the direction spectrum of the same solutions}\label{s:spectrum}

\subsection{Definitions and data}\label{s:spdata}
For $v=(a,b)$ with $\gcd(a,b)=1$ and $(a,b)\ne(0,0)$, a \emph{pair of direction $v$} in $S\subseteq[n]^2$ is an unordered pair
$\{p,q\}\subseteq S$ with $q-p=kv$ for some $k\ge1$ (so every pair has exactly one primitive direction up to sign).
For a solution, every line carries at most two points, hence the number of pairs of direction $v$ equals the number of
lines of direction $v$ carrying exactly two points.  Define
\[c_v(n)=\frac{1}{n}\,\mathbb{E}\bigl[\#\{\text{pairs of direction }v\}\bigr],\]
the expectation over the uniform distribution on the solutions of size $n$ in the sample at hand, averaged over the
class $\{(a,b),(b,a),(a,-b),(b,-a)\}$ of $v$ under the symmetries of the square.  Since every row and every column of a
solution carries exactly two points, $c_{(1,0)}=c_{(0,1)}=1$ identically, and the total over all non-axial classes
is fixed:
\begin{equation}\label{eq:total}
\sum_{v\ \text{non-axial}} n\,c_v(n)\cdot|\text{class}(v)| = \binom{2n}{2}-2n = 2n^2-3n .
\end{equation}

\emph{Data.}  We use Achim Flammenkamp's database of all known solutions (state of 31 August 2026).  It is complete
for $n\le20$ (for $n=19$: $32\,577$ solutions; $n=20$: $118\,057$) and contains only solutions with a non-trivial
symmetry for $21\le n\le57$.  The measurements were made by a second agent with code written independently of the
model below; the values quoted as ``mixed $n$'' average over all $391\,812$ solutions with weight one.  Ensemble
measurements on this database have been published since 1997 (Flammenkamp's frequency maps of cell occupation
\cite{Density}, continued by Prellberg in 2026 for $n=56,57$); the present part continues that genre with a
different observable.

\begin{center}\begin{tabular}{lrrrr}\toprule
class & mixed $n$ & $n=19$--$20$ & $n=29$--$31$ & $n=32$--$57$\\\midrule
$(1,1)$ & 0.7306 & 0.7308 & 0.7327 & 0.7325\\
$(1,2)$ & 0.5633 & 0.5638 & 0.5635 & 0.5548\\
$(1,3)$ & 0.4545 & 0.4545 & 0.4586 & 0.4526\\
$(2,3)$ & 0.4095 & 0.4097 & 0.4121 & 0.4070\\
$(1,4)$ & 0.3744 & 0.3741 & 0.3807 & 0.3793\\
$(1,5)$ & 0.3073 & 0.3074 & 0.3228 & 0.3261\\
$(2,5)$ & 0.2846 & 0.2842 & 0.3023 & 0.3049\\
$(3,4)$ & 0.3043 & 0.3039 & 0.3175 & 0.3186\\
$(1,6)$ & 0.2618 & 0.2492 & 0.2687 & 0.2800\\
$(1,7)$ & 0.2233 & 0.2071 & 0.2340 & 0.2461\\
$(2,7)$ & 0.2081 & 0.1889 & 0.2208 & 0.2335\\
$(3,5)$ & 0.2701 & 0.2609 & 0.2758 & 0.2819\\
$(4,5)$ & 0.2390 & 0.2243 & 0.2486 & 0.2575\\
$(1,8)$ & 0.1859 & 0.1692 & 0.1954 & 0.2160\\
$(3,7)$ & 0.1961 & 0.1791 & 0.2069 & 0.2209\\
\bottomrule
\end{tabular}\end{center}
Short directions are constant in $n$ to $1\%$; long directions grow with $n$ (by $23\%$ for $(3,7)$ and $28\%$ for
$(1,8)$ between $n\approx20$ and $n\ge32$); the sizes $21\le n\le28$ ($54\,604$ solutions) are omitted from the table.  A ``constant of the database'' for a long direction is therefore an average over a mixture
of sizes, and the model must be compared with the data \emph{at the same $n$}.  Symmetry does not explain the growth:
at $n=19$ and $20$, where the database is complete, the spectrum of the solutions with a half-turn symmetry
($592$ and $675$ of them) differs from that of all solutions by at most $4\%$, and for the long directions in the
\emph{downward} direction ($c_{(1,7)}$: $0.1986$ against $0.2067$ at $n=19$).

\emph{The null model.}  Under the uniform distribution on all configurations with exactly two points in every row
and column (the correct null model: a solution is such a configuration), the expected number of pairs of direction
$v$ is the number of cell pairs of direction $v$ times the probability $2(2n-3)/(n(n-1)^2)$ that a given pair of cells
in distinct rows and columns is occupied (the group $S_n\times S_n$ acts transitively on such pairs).  At $n=20$ this
gives $c^{\rm null}_{(1,1)}=1.27$, $c^{\rm null}_{(1,2)}=0.75$, $\dots$, so the solutions are depleted in the short
directions ($0.58$, $0.75$, $0.89$, $0.94$ for $(1,1)$, $(1,2)$, $(1,3)$, $(2,3)$) and slightly enriched in the long ones
($1.01$--$1.11$).  This is the pattern to be explained.

\subsection{The line model}\label{s:model}
Fix a direction $v$.  The cells of $[n]^2$ are partitioned into the lines of direction $v$ (a cell without a
neighbour in direction $v$ is a line of length $1$); let the lines have lengths $L_1,\dots,L_r$.

\begin{theorem}\label{thm:gf}
Let $2\le m\le\sum_i\min(L_i,2)$ and let $S$ be uniformly distributed on the $m$-subsets of the cells that contain at
most two cells of every line.  Then the number of such subsets is $Z_m=[t^m]\prod_{i=1}^{r}\bigl(1+L_it+\binom{L_i}{2}t^2\bigr)$
and
\[\mathbb{E}\bigl[\#\{i:\ |S\cap\ell_i|=2\}\bigr]=\frac{1}{Z_m}\sum_{i=1}^{r}\binom{L_i}{2}\,[t^{m-2}]\prod_{j\ne i}\bigl(1+L_jt+\binom{L_j}{2}t^2\bigr).\]
\end{theorem}
\begin{proof}
A subset is admissible for $v$ exactly when it takes $x_i\in\{0,1,2\}$ cells on line $i$; there are $\binom{L_i}{x_i}$
ways to choose them, independently across lines, so the admissible $m$-subsets are counted by the coefficient of $t^m$
in the product.  The expectation is $\sum_iP(x_i=2)$, and the admissible $m$-subsets with $x_i=2$ are counted by
$\binom{L_i}{2}$ times the coefficient of $t^{m-2}$ in the product over $j\ne i$.
\end{proof}

This is the canonical ensemble of hard particles in $r$ independent boxes of capacity two with degeneracies
$1,L_i,\binom{L_i}{2}$ (equivalently, a uniform placement conditioned on ``at most two per box''); the identity is
elementary.  We did not find an established name for it; the nearest umbrella term in the refereed literature is
\emph{restricted occupancy} (Freund 1956 \cite{Freund}; the unweighted equal-box case is an exercise in Comtet),
and the corresponding relaxation appears as ``capacity-two constraints'' in Prellberg's linear-programming bound for a
parity variant \cite{PrellbergCB}.  What is not elementary is that it describes the solutions of the no-three-in-line problem.

\emph{The model.}  Put $m=2n$ and let $E_v$ be the expectation of Theorem~\ref{thm:gf} for the lines of direction $v$.
The model does not know that the pairs of a solution are counted by \eqref{eq:total}: summed over all non-axial
primitive directions, $\sum_vE_v=615.3$ at $n=20$ against the true $740$, a shortfall of $17\%$.  We therefore
impose \eqref{eq:total} by one global factor, $\lambda_n=(2n^2-3n)/\sum_vE_v$ ($\lambda_{20}=1.203$, $\lambda_{30}=1.163$,
$\lambda_{40}=1.139$), and set $c^{\rm model}_v(n)=\lambda_nE_v/n$, averaged over the class of $v$.  The factor is not fitted to
the measured spectrum --- it is the same for every direction --- but it is not derived either.  Everything the model
predicts is therefore a statement about \emph{ratios} between directions.

\subsection{Results}\label{s:spresults}
\begin{center}\begin{tabular}{lrrrrrr}\toprule
class & model $20$ & data $19$--$20$ & ratio & model $30$ & data $29$--$31$ & ratio\\\midrule
$(1,1)$ & 0.7637 & 0.7308 & 1.045 & 0.7441 & 0.7327 & 1.016\\
$(1,2)$ & 0.5629 & 0.5638 & 0.998 & 0.5541 & 0.5635 & 0.983\\
$(1,3)$ & 0.4327 & 0.4545 & 0.952 & 0.4316 & 0.4586 & 0.941\\
$(2,3)$ & 0.3766 & 0.4097 & 0.919 & 0.3792 & 0.4121 & 0.920\\
$(1,4)$ & 0.3428 & 0.3741 & 0.916 & 0.3492 & 0.3807 & 0.917\\
$(1,5)$ & 0.2781 & 0.3074 & 0.905 & 0.2890 & 0.3228 & 0.895\\
$(2,5)$ & 0.2544 & 0.2842 & 0.895 & 0.2665 & 0.3023 & 0.882\\
$(3,4)$ & 0.2737 & 0.3039 & 0.901 & 0.2839 & 0.3175 & 0.894\\
$(1,6)$ & 0.2322 & 0.2492 & 0.932 & 0.2440 & 0.2687 & 0.908\\
$(1,7)$ & 0.1931 & 0.2071 & 0.932 & 0.2105 & 0.2340 & 0.900\\
$(2,7)$ & 0.1800 & 0.1889 & 0.953 & 0.1978 & 0.2208 & 0.896\\
$(3,5)$ & 0.2309 & 0.2609 & 0.885 & 0.2443 & 0.2758 & 0.886\\
$(4,5)$ & 0.2076 & 0.2243 & 0.926 & 0.2223 & 0.2486 & 0.894\\
$(1,8)$ & 0.1657 & 0.1692 & 0.979 & 0.1826 & 0.1954 & 0.934\\
$(3,7)$ & 0.1669 & 0.1791 & 0.932 & 0.1852 & 0.2069 & 0.895\\
\bottomrule\end{tabular}\end{center}
At $n=20$ the model's order of the fifteen classes coincides with the data's except for the adjacent pair $(1,6)$, $(3,5)$
(which differ by $3\%$ in the data); the model is below the data in every class but $(1,1)$, by $5$--$12\%$, and at
$n\approx30$ by $6$--$12\%$ in every class but $(1,1)$ and $(1,2)$.  The model thus reproduces the fifteen measured constants
within $12\%$ (model$/$data $0.89$--$1.05$ at $n=20$, $0.88$--$1.02$ at $n\approx30$) and their order at matched $n$ up to one
adjacent pair.

\emph{Blind test.}  The first eight classes were known to us when the model was built.  The remaining seven were
\emph{not}: their model values were written into a dated ledger (3 September 2026, 16:05 WITA) before a second agent,
who had not seen the model or the ledger, measured them on the database with independent code; all seven fell within
$2$--$10\%$ of the prediction (model$/$data at mixed $n$: $0.93$, $0.95$, $0.95$, $0.90$, $0.93$, $0.98$, $0.94$)
and in the predicted order up to pairs differing by less than $3\%$.  At matched $n=20$ the agreement for the long
directions improves (model$/$data $0.93$ for both $(1,7)$ and $(3,7)$ instead of $0.87$ and $0.85$ against the mixture).

\subsection{What fails, and what is not claimed}\label{s:spfails}
\emph{A first-order pair model fails.}  The natural first attempt --- a pair of direction $v$ survives with probability
$(1-\rho)^{L}$, $\rho$ the conditional occupancy of a further cell on its line and $L$ the number of such cells, then
renormalise --- over-depletes the short directions by a factor of two ($c_{(1,1)}=0.40$ against $0.73$) and puts $(1,2)$
above $(1,1)$.  The exclusions along a long line are far from independent; the line model succeeds because it treats
each line as a unit.

\emph{The scale is imposed.}  The model's total falls $17\%$ short of \eqref{eq:total}.  After renormalisation the
fifteen measured classes sum to $17.1$ (model) against $18.6$ (data) out of the fixed total $2n-3=37$ at $n=20$; the
model's deficit in the measured directions is exactly its excess in the tail of rare, steep directions, and no second
renormalisation can move it.  Of the shortfall, at most a third can be attributed to the row and column sums that the
model ignores (a free $2n$-subset has $705.7$ non-axial pairs in expectation, the fixed-marginal model $740$);
the rest is the model's treatment of one direction at a time, whereas a solution is constrained in all directions
simultaneously.

\emph{Drift in $n$.}  The model's values drift by $3$--$8\%$ between $n=20$ and $40$ (upward for long directions),
in the same direction as the data but not by the same amount; the constancy of the short directions in the data is
reproduced only approximately ($c^{\rm model}_{(1,1)}$: $0.764$, $0.744$, $0.731$ at $n=20,30,40$).  We found neither
a measurement nor an explanation of the growth of the long directions in the literature; the nearest theoretical
statement, Remark~1.7 of Ghosal et al.\ \cite{Ghosal}, predicts for $k=2$ an equal contribution of every dyadic scale
of the direction length --- that is, no growth.

\emph{One inversion.}  At $n=20$ the model orders $(1,6)$ above $(3,5)$ by $0.6\%$; the data order them the other way by
$4.7\%$ ($0.2492$ against $0.2609$) --- the model's largest error, $-11.5\%$, is on $(3,5)$.

\emph{Not claimed.}  No theorem about the ensemble of solutions is proved here; Theorem~\ref{thm:gf} is a statement
about a model.  No derivation of the scale, and no explanation of the growth of the long directions with $n$, is
offered.

\subsection{Prior art for the spectrum}\label{s:prior2d}
A multi-channel search (3 September 2026: direct retrieval of the primary sources; full-text search inside arXiv, zbMATH, OEIS,
Google Books and Open Library; API queries to Crossref, OpenAlex, Semantic Scholar, GitHub and the Wayback Machine)
found no published measurement of pairs per direction, or of diagonal occupation, on actual no-three-in-line solutions;
no $\kappa$-like constant; no per-direction decomposition of the Guy--Kelly count; and no model of the kind of
Section~\ref{s:model}.  The nearest objects, none of which is prior art for the results above, are: Guy and Kelly's expected number
of collinear triples of a random $2n$-subset \cite{GuyKelly}, a sum over primitive directions that is collapsed through
Euler's function and never split by direction; Flammenkamp's frequency maps of cell occupation on the same database
\cite{Density} and his ``superficial explanation'' of the shading cost of a pair of points, aggregated over symmetry
classes; Prellberg's linear-programming upper bound with capacity-two constraints on rows, columns and the two diagonal
families for a parity variant \cite{PrellbergCB} --- a relaxation, not an ensemble, and without the weights
$\binom{L}{x}$; Simkin's occupation of diagonals by direction for $n$-queens (capacity one), with an explicit independence
assumption across directions \cite{Simkin}; and a web-posted, unrefereed analysis of July 2026 \cite{Du} that measured
central directions of symmetric solutions against a random null model and reached the same qualitative conclusion
(short directions depleted).  In view of the last item the qualitative observation is not claimed as new here; what is
new is the quantitative spectrum over all pairs, its fifteen constants at matched $n$, and the model.
\emph{Limits of the search:} free-text web search was not available during it, and several primary sources were not
read (Ungar 1982, Jamison 1984--86, Kaplansky--Riordan 1946, Riordan 1958, Hall--Jackson--Sudbery--Wild 1975,
Flammenkamp's JCTA papers of 1992 and 1998, the restricted-occupancy papers of Freund), so the negative result is stated
with that caveat; the full record of channels, verdicts and confidence is public in the repository
(\texttt{docs/research/deep\_research\_14\_answer.md}).

\section{Across dimensions}\label{s:across}
Two of the parts ran the same first-moment test on the other dimension's data.  The results are placed here, where the
data live, and are recorded as indications only.

\subsection{A first-moment threshold in the cube, refuted by witnesses}\label{s:refuted}
Setting $\log\binom{n^3}{m}$ equal to the expected number of coplanar quadruples in a random
$m$-subset of $[n]^3$, computed \emph{as if the constraints were independent}, gives a threshold that
landed on exactly $2n+5$ for $n=5,\dots,12$.  Against the values then available its slack ran
$2,1,1,1,0,0$: monotonically to zero, which reads as a bound becoming tight.  It is not a bound.
The witness for $b(11)\ge28$ exceeds $2n+5=27$.

The failure has a mechanism.  The estimate treats the forbidden quadruples as independent; they
are positively correlated, since a set that has already avoided one plane avoids the next more
easily.  A first-moment threshold therefore \emph{understates}, and we verified the sign on the
companion estimate as well: the quadruple-independence threshold lies below the certified lower
bound at every $n$ from 5 to 12, with the deficit $-1,-2,-3,-3,-4,-4,-5,-4$: growing to $n=11$, receding by one at the last step.  What
looked like a bracket around the truth was one biased estimator evaluated at two coarsenings,
both leaning the same way.

We record the reading error separately, because it is not specific to this problem: a slack that
converges to zero does not distinguish a tight estimate from one about to be crossed, and the two
cannot be told apart from inside the model.

\subsection{The Guy--Kelly form of the estimate on the cube enumerations}\label{s:cubetest}
The same test as in Section~\ref{s:count} was carried out in the cube (no four coplanar) on the complete enumerations at
$n=3,4,5$ of Part I$'$ by the first solver.  There the increment of the error \emph{grows}: $+2.22$, then $+5.53$.
Two increments \emph{suggest} a wrong exponent --- and by our own standard in
Section~\ref{s:ratio}, where four exact values read as a plateau turned out to lie before the summit, two
increments decide nothing.  We record this as an indication only.  What it does mean is that the
conclusion of Section~\ref{s:class} is established in the plane and must not be carried into the cube
unexamined.  This is consistent with the
mechanism: in the plane the hard ceiling $2n$ binds and anchors the estimate, while in the cube the
corresponding bound $3n$ (Proposition~\ref{p:3n}) stands far from the achievable values and does not.

The two dimensions also share the instrument of Sections~\ref{s:lb3}, \ref{s:subspace} and \ref{s:strata4}: searching inside
symmetry strata.  In the plane, where complete material exists, the enrichment of maxima by symmetry is measured (factors of
$8$ to $73\,000$ at $n=6,\dots,11$); in the cube the same enrichment is used, and its price is recorded --- the cyclically
invariant subspace misses $b(6)$ by one point (Section~\ref{s:subspace}), and the gaps of $4$, $3$, $4$, $7$ to the 2016
contest records at $n=17,19,23,29$ measure what confinement to that subspace costs as $n$ grows (Section~\ref{s:prior3d}).

\section{Verification and reproducibility}\label{s:repro}
\emph{Witnesses.}  Every configuration reported as a lower bound in Parts I and I$'$ was checked by a program sharing no code
with the search that found it, and each checker is first given a deliberately corrupted witness that it must reject --- a check
that cannot fail confirms nothing.  For no three collinear the checker tests all $\binom{|S|}{3}$ triples by exact integer cross
products and checks that the points are distinct and inside the cube (\texttt{verify\_witness\_lines.py}).  For no four coplanar,
configurations were checked twice, by programs written independently and testing different statements: the first computes the
$3\times3$ determinant of every quadruple in exact integer arithmetic (\texttt{verify\_witness.py}); the second takes every
triple, forms the normal of the plane through it, and tests the scalar product against every remaining point, and it also checks
that the points are distinct and lie inside the cube, which the first test does not ask.  All runs reported zero violations, and
the thirteen configurations among the ancillary files were re-verified by the two scripts when this article was assembled.

\emph{Certificates and coverage.}  In Part I the single-call unsatisfiability answers carry DRAT certificates checked by
\texttt{drat-trim}; the split instances are regenerated deterministically on demand, and their case split, its completeness and
its aggregation are each verified separately, by an aggregator that builds the expected pieces from the manifest and refuses a
verdict on a missing or undecided piece (Section~\ref{s:exact3}).  The protocol \texttt{docs/VERIFICATION.md} of the repository
is written for a reader who trusts none of our code and gives the commands in order.

\emph{The plane.}  Every estimator run behind the tables of Part II is in \texttt{logs/kelly/} of the repository, including the
measurements that refuted our own earlier claims.  For Part II$'$ the measurement code, the model code with an executable
statement of Theorem~\ref{thm:gf} checked against brute force for $n\le5$, the tests written by a second agent, the two
adversarial reviews (one of which found the false claim ``the data are constant in $n$'' in an earlier draft and forced the
comparison at matched $n$), the dated ledger of predictions and the prior-art search of Section~\ref{s:prior2d} are public in
the repositories below (lemma 005 and need 008 of \texttt{lemma-atelier}).

\emph{Journals.}  Every number in this article traces to a file under \texttt{logs/} of the repository; by the rule of that
repository a number without such a reference counts as unverified, whoever wrote it.

\emph{Records.}  Programs, journals, witnesses, manifests and the verification protocol:
\url{https://github.com/iwasborninbali/saturation}; the papers, packages and witnesses are mirrored at
\url{https://github.com/iwasborninbali/no3-results}; the configurations of Part I with their per-class results are
archived in \cite{Witnesses}; the lemmas cited in Sections~\ref{s:lb3} and \ref{s:strata4} are in
\url{https://github.com/iwasborninbali/lemma-atelier}; the database of solutions is Flammenkamp's \cite{Flammenkamp}.
Earlier versions of the four parts were deposited as separate notes at Zenodo, and this article supersedes them:
\begin{itemize}\itemsep0pt
\item Part I, version 1.5: \texttt{doi:10.5281/zenodo.22273425};
\item Part I$'$, version 3.4: \texttt{doi:10.5281/zenodo.22272371};
\item Part II, version 1.9: \texttt{doi:10.5281/zenodo.22063379};
\item Part II$'$, version 1.1: \texttt{doi:10.5281/zenodo.22278951}.
\end{itemize}

\emph{Ancillary files} (directory \texttt{anc/} of the arXiv package): \texttt{verify\_witness\_lines.py} and the six witnesses
of Section~\ref{s:lb3} for $a(8)\ge94$ (two classes), $a(9)\ge116$ (two classes), $a(10)\ge138$ and $a(11)\ge164$;
\texttt{verify\_witness.py} and the seven configurations of Section~\ref{s:strata4}; and, for Part II$'$,
\texttt{direction\_spectrum\_model\_values.txt} (the model's constants at $n=20,30,40$ with the factors $\lambda_n$) and
\texttt{direction\_spectrum\_data\_by\_strata.txt} (the measured constants by $n$-stratum, with the solution counts).

\section{AI/tool-use disclosure}\label{s:disclosure}
The searches, the programs and the text of this article were produced by autonomous AI agents (Anthropic Claude) working under
the direction of the author of record, who posed the questions, chose what to compute, decided what to claim and takes
responsibility for the content.  In each part two agents worked with deliberately different program designs and audited each
other's conclusions rather than each other's code: one searched, the other verified without reading the searcher's code, so that
the verification would not inherit its blind spots; in Part II one agent worked in each dimension; in Part II$'$ one built the
model and the other measured the database with its own code and did not read the model before the blind test, and a third,
adversarial pass reviewed the claims before publication.  Several errors reported in the repository were caught that way and not
by their author, including one --- an unverified symmetry pruning --- that no certificate, witness check or coverage check would
have detected.

The work is arranged so that this should not matter.  Every claim above is meant to be checkable by a third
party without trusting the agents or the author: certificates by \texttt{drat-trim}, witnesses by brute-force
integer arithmetic, coverage by a mechanical aggregator that refuses a verdict on a missing or undecided piece,
the encoding by comparison against the direct collinearity criterion, and the model of Part II$'$ by its dated blind
predictions.  Where a claim is not checkable in that way, it is listed in the part's statement of what is not established.


\begin{thebibliography}{99}
\bibitem{GuyKelly} R.~K.~Guy and P.~A.~Kelly, The no-three-in-line problem, \emph{Canad.\ Math.\ Bull.} 11 (1968) 527--531.
\bibitem{Flammenkamp} A.~Flammenkamp, The no-three-in-line problem, \url{https://wwwhomes.uni-bielefeld.de/achim/no3in/readme.html}.
\bibitem{Density} A.~Flammenkamp, Frequency distributions of points of no-three-in-line configurations,
\url{https://wwwhomes.uni-bielefeld.de/achim/no3in/density.html} (1997--2026).
\bibitem{Survey} The general position problem: a survey, arXiv:2501.19385 (2025).
\bibitem{Prellberg} T.~Prellberg, Constraint satisfaction programming for the no-three-in-line problem, arXiv:2602.07751;
\emph{J.\ Combin.\ Theory Ser.~A}, to appear.
\bibitem{PrellbergCB} T.~Prellberg, No-three-in-line sets on the checkerboard grid, arXiv:2605.09215 (May 2026).
\bibitem{Voutier} P.~M.~Voutier, On the Guy--Kelly conjecture for the no-three-in-line problem, arXiv:2603.00215 (2026).
\bibitem{Kaplan} D.~Eppstein, Random no-three-in-line sets (summary of a talk by N.~Kaplan),
\url{https://11011110.github.io/blog/2018/11/10/random-no-three.html}.
\bibitem{Ghosal} A.~Ghosal, R.~Goenka, A.~Grebennikov, P.~Keevash, M.~Kwan and H.~T.~Pham, No-$(k+1)$-in-line problem for
$k\ge3$, arXiv:2607.05255 (July 2026).
\bibitem{Simkin} M.~Simkin, The number of $n$-queens configurations, \emph{Adv.\ Math.} 427 (2023) 109127.
\bibitem{Freund} J.~E.~Freund, Restricted occupancy theory --- a generalization of Pascal's triangle, \emph{Amer.\ Math.\
Monthly} 63 (1956) 20--27.
\bibitem{Du} J.~Du, no3inline-rigidity (GitHub repository aujurd22/no3inline-rigidity, July 2026), unrefereed.
\bibitem{PorWood} A.~P\'or and D.~R.~Wood, No-Three-in-Line-in-3D, \emph{Algorithmica} 47 (2007) 481--488;
preliminary version in Proc.\ Graph Drawing 2004, Lecture Notes in Comput.\ Sci.\ 3383, Springer.
\bibitem{Kissat} A.~Biere, K.~Fazekas, M.~Fleury and M.~Heisinger, CaDiCaL, Kissat, Paracooba, Plingeling and
Treengeling entering the SAT Competition 2020, Proc.\ SAT Competition 2020, University of Helsinki, 2020.
\bibitem{DratTrim} N.~Wetzler, M.~J.~H.~Heule and W.~A.~Hunt, Jr., DRAT-trim: efficient checking and trimming
using expressive clausal proofs, Proc.\ SAT 2014, Lecture Notes in Comput.\ Sci.\ 8561, Springer, 2014.
\bibitem{Glucose} G.~Audemard and L.~Simon, Predicting learnt clauses quality in modern SAT solvers,
Proc.\ IJCAI 2009.
\bibitem{OEIS} OEIS Foundation Inc., The On-Line Encyclopedia of Integer Sequences, \url{https://oeis.org}:
sequences A399138, A000755, A000769, A280538.
\bibitem{OEISA280537} OEIS Foundation Inc., Sequence A280537 (M.~Scheuermann, 2017), with the
January 2017 comment recording lower bounds up to its $a(17)\ge42$ (our $b(17)$) and the sourcing of its $a(7)$, $a(8)$
(our $b(7)$, $b(8)$); \url{https://oeis.org/A280537}.  See also A280538 (numbers of optimal solutions) and the
references there: E.~Pegg~Jr.\ (Wolfram Demonstrations; Mathematics Stack Exchange), T.~Sillke
(rec.puzzles, 1992), W.~M\"ohres (exhaustive search for $n=6$, private communication, 2016).
\bibitem{AZsPCs} Al Zimmermann's Programming Contests, \emph{Non-Coplanar Points}, March--June
2016; problem statement and final report at \url{http://azspcs.com/Contest/Tetrahedra}.  Scores
are squares of set sizes; per-size bests decoded accordingly (winner T.~Foden).
\bibitem{Witnesses} A.~Kudriashov, Witness configurations and lower bounds for the no-three-in-line problem in the
$n\times n\times n$ grid (OEIS A399138), Zenodo, 2026, \url{https://doi.org/10.5281/zenodo.22271375}; source
\url{https://github.com/iwasborninbali/a399138-witnesses}.
\end{thebibliography}
\end{document}